\documentclass[amscd, amssymb, verbatim]{amsart}[12pt]

\usepackage{graphicx,import}
\usepackage{fullpage}
\usepackage{hyperref}
\usepackage{placeins}
\usepackage{color}
\usepackage[utf8x]{inputenc}

\usepackage{todonotes}

\theoremstyle{plain}
\newtheorem{lemma}{Lemma}[section]
\newtheorem{corollary}[lemma]{Corollary}
\newtheorem{theorem}[lemma]{Theorem}
\theoremstyle{definition}

\newtheorem{definition}[lemma]{Definition}
\newtheorem{example}[lemma]{Example}
\newtheorem{question}[lemma]{Question}
\newtheorem{remark}[lemma]{Remark}

\newcommand{\D}{\mathbb{D}}

\newcommand{\R}{\mathbb{R}}
\newcommand{\cF}{\mathcal{F}}
\newcommand{\SC}{\mathrm{SC}}

\newcommand{\area}{\mathrm{area}}

\newcommand{\wn}{\mathrm{wn}}

\begin{document}
\title{Sweeping costs of planar domains}
\author{Brooks Adams}\email{brooksad@rams.colostate.edu}
\author{Henry Adams}\email{adams@math.colostate.edu}
\author{Colin Roberts}\email{robertsp@rams.colostate.edu}
\begin{abstract}
Let $D$ be a Jordan domain in the plane. We consider a pursuit-evasion, contamination clearing, or sensor sweep problem in which the pursuer at each point in time is modeled by a continuous curve, called the \emph{sensor curve}. Both time and space are continuous, and the intruders are invisible to the pursuer. Given $D$, what is the shortest length of a sensor curve necessary to provide a sweep of domain $D$, so that no continuously-moving intruder in $D$ can avoid being hit by the curve? We define this length to be the \emph{sweeping cost} of $D$. We provide an analytic formula for the sweeping cost of any Jordan domain in terms of the geodesic Fr\'{e}chet distance between two curves on the boundary of $D$ with non-equal winding numbers. As a consequence, we show that the sweeping cost of any convex domain is equal to its width, and that a convex domain of unit area with maximal sweeping cost is the equilateral triangle.
\end{abstract}
\maketitle

\section{Introduction}

Let $D$ be a Jordan domain, i.e.\ the homeomorphic image of a disk in the plane. Suppose that continuously-moving intruders wander in $D$. You and a friend are each given one end of a rope, and your task is to drag this rope through the domain $D$ in such a way so that every intruder is eventually intersected or caught by the rope. What is the shortest rope length you need in order to catch every possible intruder? We refer to such a continuous rope motion as a \emph{sweep} of $D$ (Figure~\ref{fig:sweep}), and we refer to the length of the shortest such possible rope as the \emph{sweeping cost} of $D$.

\begin{figure}[h]
	\def\svgwidth{5in}
	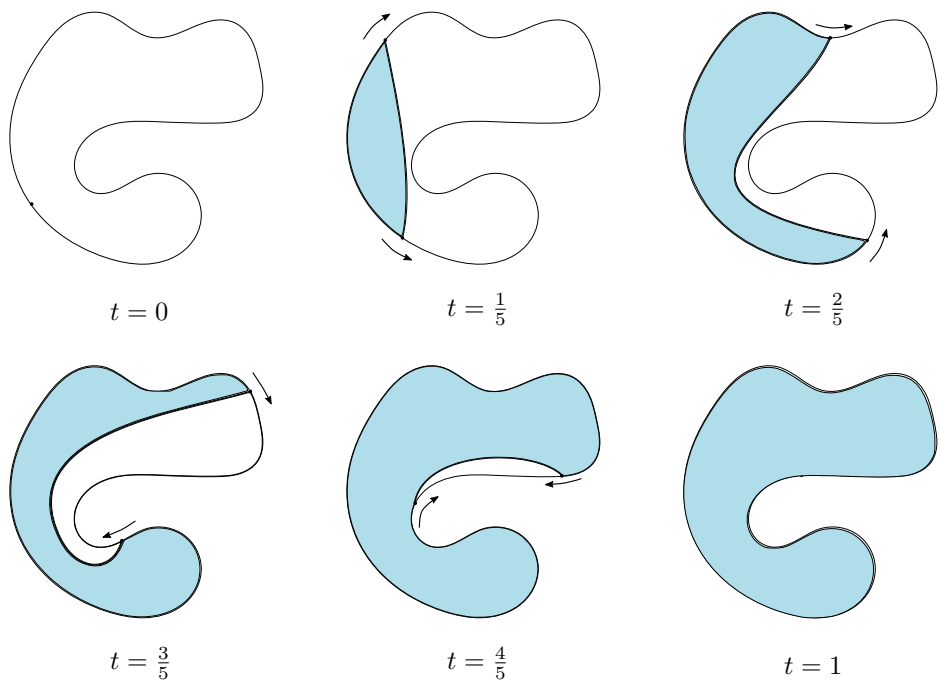
\caption{An example sweep of a domain in the plane. A time $t=0$ the entire domain is contaminated, and at time $t=1$ the clearing sweep is complete.}
\label{fig:sweep}
\end{figure}

The problem we consider is only one example of a wide variety of interesting pursuit-evasion problems; see Section~\ref{sec:related} for a brief introduction or \cite{chung2011search}, for example, for a survey. It is a pursuit-evasion problem in which both space and time are continuous, the pursuer is modeled at each point in time by a continuous curve, the intruder has complete information about the pursuer's location and its planned future movements, and the pursuer has no knowledge of the intruder's movements. Our problem can also be phrased as a contamination-clearing task, in which one must find the shortest rope necessary to clear domain $D$ of a contaminant which, when otherwise unrestricted by the rope, moves at infinite speed to fill its region.

As first examples, the sweeping cost of a disk is equal to its diameter, and the sweeping cost of an ellipse is equal to the length of its minor axis. These computations follow from Theorem~\ref{thm:area-lower-bound}, in which we prove that the sweeping cost of domain $D$ is at least as large as the shortest area-bisecting curve in $D$.

One motivation for considering a pursuer which is a rope, or a continuous curve at each point in time, is the context of mobile sensor networks. Suppose there is a large collection of disk-shaped sensors moving inside a planar domain, as considered in \cite{de2006coordinate,adams2015evasion}. What is the minimal number of sensors needed to clear this domain of all possible intruders? If $n$ is the number of sensors, and $\frac{1}{n}$ is the diameter of each sensor, then as $n\to\infty$ an upper bound for the number of sensors needed is given by the sweeping cost.

\begin{figure}[h]
	\def\svgwidth{5in}
	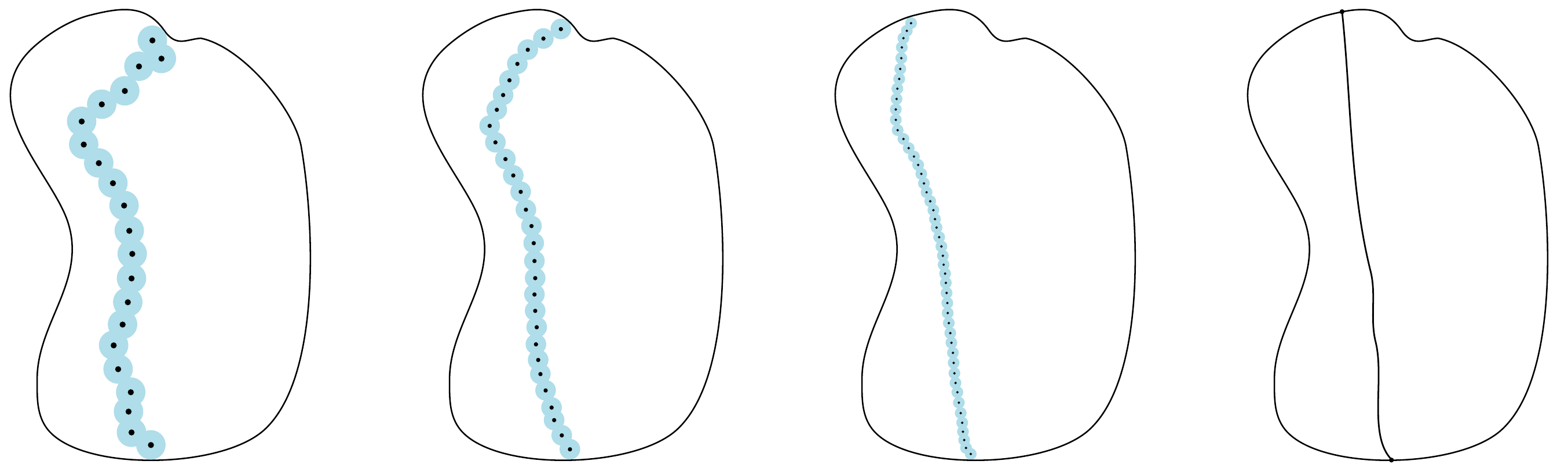
\label{fig:robots_to_line}
\end{figure}

As our main result, in Theorem~\ref{thm:frechet} we provide an analytic formula for the sweeping cost of an arbitrary Jordan domain $D$. Indeed, the sweeping cost of $D$ is equal to the infimum, taken over all pairs of curves in the boundary of $D$ whose concatenation wraps a nontrivial number of times around the boundary, of the geodesic Fr\'{e}chet distance between the two curves. The geodesic Fr\'{e}chet distance differs from the standard Fr\'{e}chet distance in that the distance between two points in $D$ is not their Euclidean distance, but instead the length of the shortest path between them in $D$. Our Theorem~\ref{thm:frechet} is also closely related to the geodesic width between two curves \cite{efrat2002new}.

Using our main result, we prove in Theorem~\ref{thm:sc=width} that the sweeping cost of a convex domain $D$ is equal to the width of $D$. As a consequence, it follows that the sweeping cost of a polygonal convex domain with $n$ vertices can be computed in time $O(n)$ and space $O(n)$ using the \emph{rotating calipers} technique \cite{toussaint1983solving}. Furthermore, it follows from \cite{pal1921minimumproblem,cerdan2006comparing} that a convex domain of unit area with the maximal possible sweeping cost---i.e.\ the most expensive convex domain to clear with a rope---is the equilateral triangle.

An intriguing open question motivated by our work is the following (Question~\ref{ques:weak-strong}). Given a Jordan domain $D$ and two continuous injective curves $\alpha,\beta$ with image in the boundary $\partial D$, is the weak geodesic Fr\'{e}chet distance between $\alpha$ and $\beta$ equal to the strong geodesic Fr\'{e}chet distance? The weak version of the Fr\'{e}chet distance allows $\alpha$ and $\beta$ to be reprarametrized non-injectively, whereas the strong version does not. 

We review related work in Section~\ref{sec:related}, state our problem of interest in Section~\ref{sec:preliminaries}, and describe some basic properties of the sweeping cost in Section~\ref{sec:properties}. In Section~\ref{sec:lower-bound} we provide a lower bound on the sweeping cost in terms of shortest area-bisecting curves. We provide analytic formulas for the sweeping costs of Jordan and convex domains in Sections~\ref{sec:jordan} and \ref{sec:convex}, and in Section~\ref{sec:extremal} we deduce that the sweeping cost of a unit-area convex domain is maximized by the equilateral triangle. The conclusion describes related problems of interest, and the appendix contains two technical lemmas and their proofs.

\section{Related work}\label{sec:related}

A wide variety of pursuit-evasion problems have appeared in the mathematics, computer science, engineering, and robotics literature; see \cite{chung2011search} for a survey. Space can modeled in a discrete fashion, for example by a graph \cite{alspach2006searching,bienstock1991graph}, or as a continuous domain in Euclidean space as we consider here. Time can similarly be discrete (turn-based) or continuous, as is our case. See \cite{alonso1992lion,beveridge2015two,chin2010detection,cortes2002coverage,liu2005mobility,parsons1978pursuit} for a selection of such problems.

A further important distinction in a pursuit-evasion problem is whether information is complete (pursuers and intruders know each others' locations), incomplete (pursuers and intruders are invisible to each other), or somewhere in-between. Our problem can be considered as one in which the pursuer has no knowledge of the intruders' movements, whereas the intruders have complete knowledge of the pursuer's current position and future movements. In other words, the pursuer must catch every possible intruder. Evasion problems in which the pursuer has no information and the intruders have complete information can equivalently be cast as contamination-clearing problems, see for example \cite{berger2009many,de2006coordinate,adams2015evasion}. Indeed, the contaminated region of the domain at a particular time includes all locations where an intruder could currently be located, and the uncontaminated region is necessarily free of intruders. It is the task of the pursuer to clear the entire domain of contamination, so that no possible intruders could remain undetected.

The paper \cite{efrat2002new} introduces the \emph{geodesic width} between two polylines, a notion that is very relevant for our problem. The geodesic width between two curves $\alpha$, $\beta$ is the same as the strong geodesic Fr\'{e}chet distance between them (Definition~\ref{def:geodesic-frechet}) when domain $D$ is chosen to be a region with boundary consisting of curves $\alpha$, $\beta$, and the two shortest paths connecting the endpoints of $\alpha$ and $\beta$. If curves $\alpha$ and $\beta$ are polylines with $n$ vertices in total, then \cite{efrat2002new} gives an $O(n^2\log n)$ algorithm for computing the geodesic width between them. The goal of our paper is instead to rigorously prove an analytic formula for the sweeping cost of a domain, Theorem~\ref{thm:frechet}, that is closely related to geodesic widths. Indeed, the right hand side of \eqref{eq:frechet} in Theorem~\ref{thm:frechet} is unchanged if we replace the geodesic distance between two curves (Definition~\ref{def:geodesic-dist-curves}) with the weak geodesic Fr\'{e}chet distance between them (Definition~\ref{def:geodesic-frechet}). The paper \cite{efrat2002new} also studies sweeps of planar domains by piecewise linear curves in which the cost of a sweep is not equal to a length, but instead to the number of vertices or joints in the curve.

Related notions to the geodesic width include the isotopic Fr\'{e}chet distance \cite{chambers2011isotopic} and the minimum deformation area \cite{wangmeasuring} between two curves $\alpha$ and $\beta$. Whereas the geodesic width considers deformations between $\alpha$ and $\beta$ such that no two intermediate curves intersect, this restriction is not present for the isotopic Fr\'{e}chet distance, which can therefore be defined between intersecting curves. The paper \cite{wangmeasuring} considers a distance between two curves on a 2-manifold which is instead an area: the minimal total surface area swept out by any deformation between the two curves. If the curves are piecewise linear in the plane, have $n$ total vertices, and have $I$ intersection points, then \cite{wangmeasuring} gives an $O(n+I^2\log n)$ algorithm to compute the minimum deformation area between them.

\section{Preliminaries and notation}\label{sec:preliminaries}

Let $d\colon\R^2\times\R^2\to\R$ denote the Euclidean metric on $\R^2$. The distance between two subsets $X,Y\subseteq\R^2$ is defined as $d(X,Y)=\inf\{d(x,y)~|~x\in X\mbox{ and }y\in Y\}$. We denote the closure of a set $X\subseteq\R^2$ by $\overline{X}$.

\subsection*{Jordan domains and geodesics}

Let $D\subseteq\R^2$ be a Jordan domain, i.e.\ the homeomorphic image of a closed disk in $\R^2$. It follows that $D$ is compact and simply-connected, and its boundary $\partial D$ is a topological circle. Given a point $x\in D$, we let $B(x,\epsilon)=\{y\in D~|~d(x,y)<\epsilon\}$ denote the open ball about $x$ in $D$. We denote the $\epsilon$-offset of a set $X\subseteq D$ by $B(X,\epsilon)=\cup_{x\in X}B(x,\epsilon)$. Given a subset $X\subseteq D$, we define its boundary as $\partial X=\overline{X}\cap\overline{\R^2\setminus X}$.

We refer the reader to \cite{burago2001course} for the basics of geodesic curves and distances. The length of a continuous path $\gamma\colon[a,b]\to D$ is defined as in \cite[Defintion~2.3.1]{burago2001course}; we denote this length by $L(\gamma)$. Curve $\gamma$ is said to be \emph{rectifiable} if $L(\gamma)<\infty$. Domain $D$ has a \emph{length structure} (\cite[Section~2.1]{burago2001course}) in which all continuous paths are admissible, and the length is given by the function $L$. 
The associated \emph{geodesic metric} $d_L\colon D\times D\to\R$, also known as a \emph{path-length} or \emph{intrinsic metric}, is
\[ d_L(x,y)=\inf\{L(\gamma)~|~\gamma\colon[a,b]\to D\mbox{ is continuous with }\gamma(a)=x,\ \gamma(b)=y\}. \]
The precise definition of a geodesic, or length-minimizing curve in $D$, is given in \cite[Definition~2.5.27]{burago2001course}. Since $D$ is a Jordan domain, it follows from \cite{bourgin1989shortest,bishop2005intrinsic} that each pair of points in $D$ is joined by a unique shortest geodesic in $D$.

\begin{definition}\label{def:geodesic-dist-curves}
Let $D\subseteq \R^2$ be a Jordan domain. We define the geodesic distance between two curves $\alpha,\beta\colon [a,b]\to D$ to be
\[ d_L(\alpha,\beta)=\max_{t\in[a,b]}d_L(\alpha(t),\beta(t)). \]
\end{definition}

\subsection*{Fr\'{e}chet and geodesic Fr\'{e}chet distances}

The \emph{Fr\'{e}chet distance} is a measure of similarity between two curves $\alpha, \beta \colon [0,1] \to \R^2$. One application of the Fr\'{e}chet distance is in handwriting input recognition for a computer \cite{sriraghavendra2007frechet}: in order to properly tell which letters a user has written, the machine must determine which curves (representing letters) are the most similar. Other notions of distance, such as the Hausdorff distance between the images of the curves, are not necessarily sensitive enough for this task.

The intuition behind the Fr\'{e}chet distance is that you are walking along path $\alpha$, your dog is walking along path $\beta$, and you want to know how long of a leash you need. There are two notions, namely the \emph{weak Fr\'{e}chet distance} and the \emph{strong Fr\'{e}chet distance}. In the weak case, you and your dog are allowed to backtrack along your respective paths, but in the strong case backtracking is forbidden. In general these two distances need not be equal.  

\begin{definition}
Let $\alpha, \beta \colon [0,1] \to \R^2$ be continuous curves. Then the \emph{weak (resp.\ strong) Fr\'{e}chet distance} between $\alpha$ and $\beta$ is
\[d_{\mbox{Fr\'{e}chet}}(\alpha,\beta) = \inf_{a,b} \max_{t \in [0,1]} \left\{d\left(\alpha(a(t),\beta(b(t)\right)\right\},\]
where the infimum is taken over all continuous $a, b \colon [0,1] \to [0,1]$ which are surjective (resp.\ bijective).
\end{definition}

If $D\subseteq \R^2$ is a Jordan domain and $\alpha, \beta \colon [0,1] \to D$ are two curves, then we can consider a variant of the Fr\'{e}chet distance in which the Euclidean metric $d$ is replaced with the geodesic metric $d_L$.

\begin{definition}\label{def:geodesic-frechet}
Let $D\subseteq \R^2$ be a Jordan domain, and let $\alpha, \beta \colon [0,1] \to D$ be continuous curves. Then the \emph{weak (resp.\ strong) geodesic Fr\'{e}chet distance} between $\alpha$ and $\beta$ is
\[d_{\mbox{geodesic Fr\'{e}chet}}(\alpha,\beta) = \inf_{a,b} \max_{t \in [0,1]} \left\{d_L\left(\alpha(a(t),\beta(b(t)\right)\right\},\]
where the infimum is taken over all continuous $a, b \colon [0,1] \to [0,1]$ which are surjective (resp.\ bijective).
\end{definition}

\subsection*{Sensor curves}

Let $I=[0,1]$ be the unit interval. We define a sensor curve to be a time-varying rectifiable curve in $D$.

\begin{definition}\label{def:sensor-curve}
A \emph{sensor curve} is a continuous map $f\colon I\times I\to D$ such that
\begin{enumerate}
\item[(i)] each curve $f(\cdot,t)\colon I\to D$ is rectifiable and injective for $t\in(0,1)$,
\item[(ii)] $f(I,0)$ and $f(I,1)$ are each (possibly distinct) single points in $\partial D$, and
\item[(iii)] $f(s,t)\in\partial D$ implies $s\in\{0,1\}$ or $t\in\{0,1\}$.
\end{enumerate}
\end{definition}
\noindent 
We think of the first input $s$ as a spatial variable and of the second $t$ as a temporal variable; in particular $f(I,t)$ is the region covered by the curve of sensors at time $t$. Assumption (ii) states that the images of the sensor curve at times 0 and 1 are single points, and assumption (iii) implies that (apart from times 0 and 1) only the boundary of the sensor curve intersects $\partial D$. We define the length of a sensor curve to be
\[ L(f)=\max_{t\in I}\ L(f(\cdot,t)). \]

An \emph{intruder} is a continuous path $\gamma\colon I\to D$. We say that an intruder is caught by a sensor curve $f$ at time $t$ if $\gamma(t)\in f(I,t)$. A path $\gamma\colon[0,t]\to D$ such that $\gamma(t')\notin f(I,t')$ for all $t'\in[0,t]$ is called an \emph{evasion path}. Sensor curve $f$ is a \emph{sweep} if every continuously moving intruder $\gamma\colon I\to D$ is necessarily caught at some time $t$, or equivalently, if no evasion path over the full time interval $I$ exists.\footnote{Our definition is similar to the graph-based definition in \cite[Definition~2.1]{alspach2006searching}.}

The following notation will prove convenient. Fix a sensor curve $f$. We let $C(t)\subseteq D$ be the \emph{contaminated region} at time $t$, and we let $U(t)\subseteq D$ be the \emph{uncontaminated region} at time $t$. More precisely,
\[ C(t)=\{x\in D~|~\exists~\gamma\colon[0,t]\to D \mbox{ with }\gamma(t)=x\mbox{ and }\gamma(t')\notin f(I,t')~\forall t'\in[0,t]\}\quad\mbox{and}\quad U(t)=D\setminus C(t). \]
Note that sensor curve $f$ is a sweep if and only if $C(1)=\emptyset$, or equivalently $U(1)=D$.

\begin{definition}
Let $\cF(D)$ be the set of all sensor curve sweeps of $D$. The \emph{sweeping cost} of $D$ is
\[ \SC(D)=\inf_{f\in\cF(D)}\ L(f).\]
\end{definition}

\begin{remark}
The results of Sections~\ref{sec:properties}--\ref{sec:lower-bound} hold even if assumptions (ii) and (iii) in Definition~\ref{def:sensor-curve} are removed.
\end{remark}


\section{Properties of sensor sweeps}\label{sec:properties}

We now prove some basic properties of sensor sweeps and the contaminated and uncontaminated regions.

\begin{lemma}\label{lem:connected}
If $x$ and $x'$ are in the same path-connected component of $D\setminus f(I,t)$, then $x\in U(t)$ if and only if $x'\in U(t)$.
\end{lemma}

\begin{proof}
Suppose for a contradiction that $x\in U(t)$ but $x'\notin U(t)$. Since $x'\in C(t)$, there exists an evasion path $\gamma : [0,t] \to D$ with $\gamma(t)=x'$. Since $x$ and $x'$ are in the same path-connected component, there exists a path $\beta\colon I \to D\setminus f(I,t)$ with $\beta(0)=x$ and $\beta(1)=x'$.

Note that $\beta(I)$ and $f(I,t)$ are compact, since they are each a continuous image of the compact set $I$. As any metric space is normal, there exist disjoint neighborhoods containing $\beta(I)$ and $f(I,t)$. Because $f$ is uniformly continuous (it is a continuous function on a compact set), we can choose $\delta_1 > 0$ such that $f(I,[t-\delta_1,t])$ remains in this open neighborhood disjoint from $\beta(I)$, giving
\begin{equation}\label{eq:beta}
\beta(I) \cap f(I,[t-\delta_1,t])=\emptyset.
\end{equation}    

Since metric space $D$ is normal, there exist disjoint neighborhoods containing $\gamma(t)=x'$ and $f(I,t)$. Since $\gamma$ is continuous and $f$ is uniformly continuous, we can choose $\delta_2 > 0$ such that $\gamma([t-\delta_2,t])$ and $f(I,[t-\delta_2,t])$ remain in these disjoint neighborhoods, giving
\begin{equation}\label{eq:gamma}
\gamma([t-\delta_2,t]) \cap f(I,[t-\delta_2,t])=\emptyset
\end{equation}

Let $\delta=\min\{\delta_1,\delta_2\}$. Using \eqref{eq:beta} and \eqref{eq:gamma} we can define an evasion path $\gamma\colon[0,t]\to D$ with $\gamma(t)=x$. Indeed, let
\[ \gamma(t')=\begin{cases}
\gamma(t')&\mbox{if }t'\le t-\delta\\
\gamma(2t'-t+\delta)&\mbox{if }t-\delta<t'\le t-\frac{\delta}{2}\\
\beta(\frac{2}{\delta}(t'-t+\frac{\delta}{2}))&\mbox{if }t-\frac{\delta}{2}<t'\le t.
\end{cases} \]
This contradicts the fact that $x\in U(t)$.
\end{proof}

\begin{lemma}\label{lem:closed-open}
 For all $t\in I$, the set $U(t)$ is closed and the set $C(t)$ is open in $D$.
\end{lemma}

\begin{proof}
Suppose $x\in C(t)$. Since $f(I,t)$ is closed and $x\notin f(I,t)$, there exists some $\varepsilon > 0$ such that $B(x,\varepsilon) \cap f(I,t) = \emptyset.$ Note all $x'\in B(x,\varepsilon)$ are in the same path-connected component of $D\setminus f(I,t)$ as $x$ via a straight line path. Hence Lemma~\ref{lem:connected} implies $B(x,\varepsilon)\subseteq C(t)$, showing $C(t)$ is open in $D$. It follows that $U(t)=D\setminus C(t)$ is closed in $D$.
\end{proof}



\begin{lemma}\label{lem:homeo}
If $h\colon D\to h(D)$ is a homeomorphism onto its image $h(D)\subseteq \R^2$, then a sensor curve $f\colon I\times I\to D$ is a sweep of $D$ if and only if sensor curve $hf\colon I\times I\to h(D)$ is a sweep of $h(D)$.
\end{lemma}

\begin{figure}[h]
	\def\svgwidth{3in}
	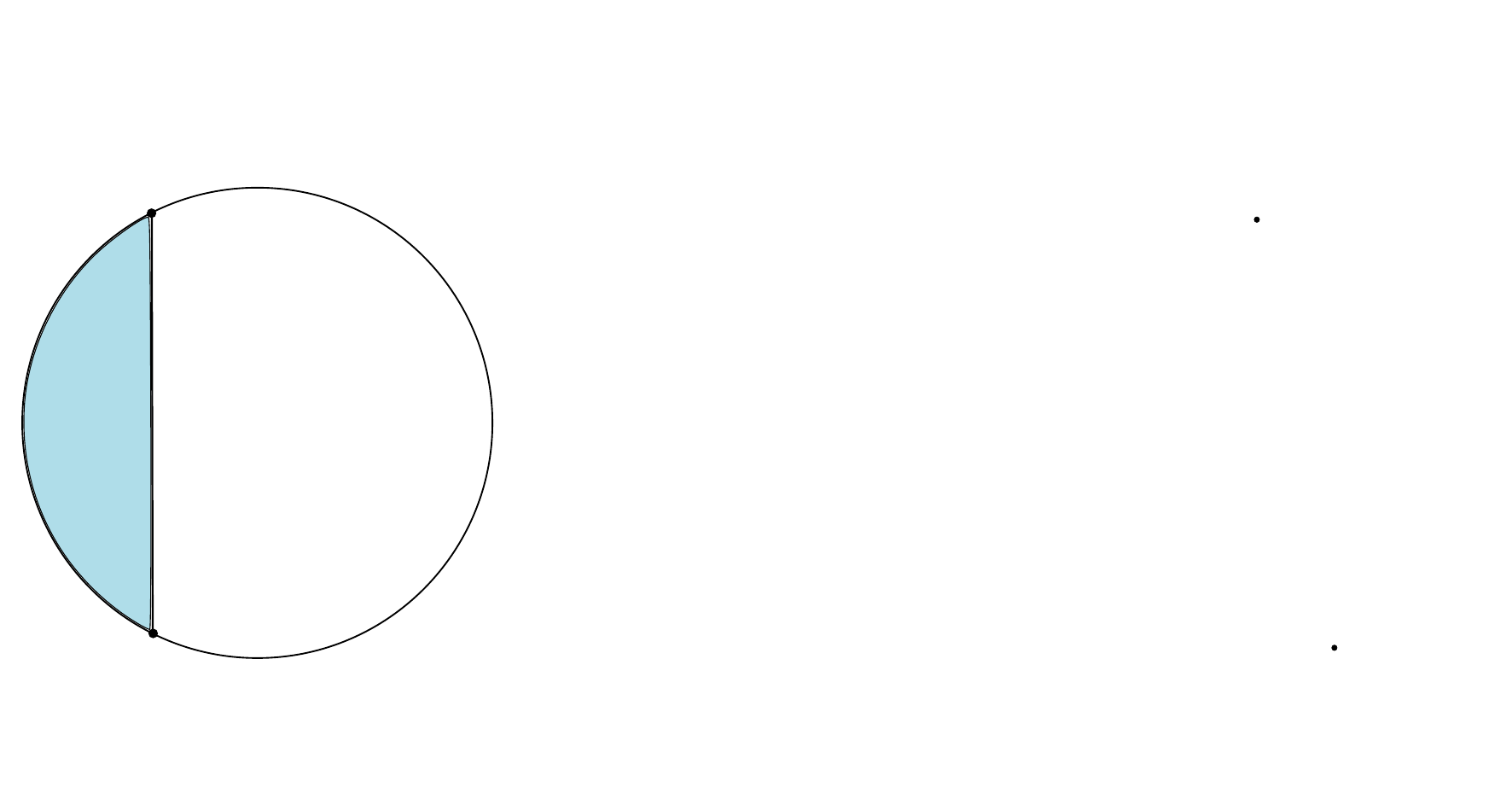
\caption{A homeomorphism $h\colon D\to h(D)$.}
\label{fig:homeomorphism}
\end{figure}

\begin{proof}
Note that if $\gamma\colon I\to D$ is an evasion path for $f$, then $h\gamma\colon I\to h(D)$ is an evasion path for $hf$. Conversely, if $\gamma\colon I\to h(D)$ is an evasion path for $hf$, then $h^{-1}\gamma\colon I\to D$ is an evasion path for $f$.
\end{proof}

\section{A lower bound on the sweeping cost}\label{sec:lower-bound}

In this section we prove that the sweeping cost of a Jordan domain is at least as large as the length of the shortest area-bisecting curve. The first lemma is a version of the intermediate value theorem with slightly relaxed hypotheses.

\begin{lemma}\label{lem:gen-intermediate-value}
If $f\colon[a,b]\to\R$ is upper semi-continuous and left continuous and if $f(a)<u<f(b)$, then there exists some $c\in(a,b)$ with $f(c)=u$.
\end{lemma}

\begin{proof}
Let $S$ be the set of all $x\in(a,b)$ with $f(x)< u$. Then $S$ is nonempty since $a\in S$, and $S$ is bounded above by $b$. Hence by the completeness of $\R$, the supremum $c=\sup S$ exists. We claim that $f(c)=u$.

Let $\epsilon>0$. Since $f$ is left-continuous, there is some $\delta>0$ such that $|f(x)-f(c)|<\epsilon$ whenever $x\in(c-\delta,c]$. By the definition of supremum, there exists some $y\in(c-\delta,c]$ that is contained in $S$, giving $f(c)<f(y)+\epsilon<u+\epsilon$. Since this is true for all $\epsilon>0$, it follows that $f(c)\le u$.

It remains to show $f(c)\ge u$. Let $\epsilon>0$. Since $f$ is upper semi-continuous, there exists a $\delta>0$ such that $f(c)>f(x)-\epsilon$ whenever $x\in(c-\delta,c+\delta)$. Let $y\in(c,c+\delta)$ and note that $y\notin S$, giving $f(c)>f(y)-\epsilon\ge u-\epsilon$. It follows that $f(c)\ge f(u)$.
\end{proof}

If $S\subseteq \R^2$ is a measurable set, then we let $\area(S)$ denote its area.

\begin{lemma}\label{lem:uppersemi-left-cont}
The function $\area(U(t))$ is upper semi-continuous and left continuous.
\end{lemma}

\begin{proof}
Let $t_0\in I$. We will show that $\area(U(t))$ is right upper semi-continuous and left continuous at $t_0$, which implies the function is both upper semi-continuous and left continuous.

For right upper semi-continuity, note for $t\ge t_0$ we have
\begin{equation}\label{eq:contain}
U(t)\subseteq U(t_0)\cup f(I,[t_0,t]).
\end{equation}
Since sensor curve $f\colon I\times I\to D$ is a continuous function on a compact domain, it is also uniformly continuous. Hence for all $\epsilon>0$ there exists some $\delta$ such that
\begin{equation}\label{eq:in-ball}
f(I,[t_0,t_0+\delta])\subseteq B(f(I,t_0),\epsilon).
\end{equation}
It follows that for all $t\in[t_0,t_0+\delta]$ we have
\begin{align*}
\area(U(t))-\area(U(t_0))&\le\area(f(I,[t_0,t]))&&\mbox{by \eqref{eq:contain}}\\
&\le\area(B(f(I,t_0),\epsilon))&&\mbox{by \eqref{eq:in-ball}}\\
&\le 2L(f(I,t_0))\epsilon+\pi\epsilon^2,
\end{align*}
where the last inequality is by a result of Hotelling (see for example \cite[Equation~(2.1)]{johnstone1989hotelling}). Hence $\area(t)$ is right upper semi-continuous.

To see that $\area(U(t))$ is left continuous at $t_0\in I$, we must show that for all sequences $\{s_i\}$ with $0\le s_i\le t_0$ and $\lim_i s_i=t_0$, we have $\lim_i\area(U(s_i))=\area(U(t_0))$.
We claim
\begin{equation}\label{eq:containments}
U(t_0)\setminus f(I,t_0)\subseteq\liminf_i U(s_i)\subseteq\limsup_i U(s_i)\subseteq U(t_0),
\end{equation}
where the middle containment is by definition. 
We now justify the first and last containment.

To prove $U(t_0)\setminus f(I,t_0)\subseteq\liminf_i U(s_i)$ it suffices to show that for any $x \in U(t_0) \setminus f(I,t_0)$ there exists an $\epsilon >0$ such that $x \in U(t_0 - \delta)$ for all $\delta \in [0,\epsilon)$.  Fix $\epsilon$ such that $x \notin f(I,t)$ for $t \in (t_0-\epsilon, t_0]$. Suppose for a contradiction that $x\in C(t_0-\delta)$ for some $\delta \in [0,\epsilon)$.  Hence there exists an evasion path $\gamma \colon [0,t_0-\delta] \to D$ with $\gamma(t_0-\delta) = x$. It is possible to extend $\gamma$ to an evasion path $\tilde{\gamma}\colon[0,t_0]\to D$ defined by
\[ \tilde{\gamma}(t)=\begin{cases}
\gamma(t)&\mbox{if }t\in[0,t_0-\delta]\\
x&\mbox{if }t\in(t_0-\delta,t_0].
\end{cases} \]
This contradicts the fact $x \in U(t_0)$, thus giving the first containment.

We now show $\limsup_i U(s_i)\subseteq U(t_0)$. If $x\notin U(t_0)$, then there exists an evasion path $\gamma\colon[0,t_0]\to D$ with $\gamma(t_0)=x$. Since $C(t_0)$ is open by Lemma~\ref{lem:closed-open}, there exists some $\delta>0$ such that $B(x,\delta)\subseteq C(t_0)$, and hence $B(x,\delta)\cap f(I,t_0)=\emptyset$. Since $f$ is uniformly continuous, there is some $\epsilon_1>0$ sufficiently small with $B(x,\delta)\cap f(I,[t_0-\epsilon_1,t_0])=\emptyset$, and since $\gamma$ is continuous there is some $\epsilon_2>0$ with $\gamma([t_0-\epsilon_2,t_0])\subseteq B(x,\delta)$. Let $\epsilon=\min\{\epsilon_1,\epsilon_2\}$. Reparametrize $\gamma$ to get a continuous curve $\tilde{\gamma}\colon[0,t_0]\to D$ with
\[ \tilde{\gamma}(t)=\begin{cases}
\gamma(t)&\mbox{if }x\in[0,t_0-\epsilon)\\
\gamma(t)\in B(x,\delta)&\mbox{if }t\in[t_0-\epsilon,t_0-\epsilon/2)\\
\gamma(t)=x&\mbox{if }t\in[t_0-\epsilon/2,t_0].
\end{cases} \]
The evasion path $\tilde{\gamma}$ shows $x\notin\limsup_i U(s_i)$, giving the third containment and finishing the proof of \eqref{eq:containments}.

Set $f(I,t_0)$ has Lebesgue measure zero since curve $f(\cdot,t_0)$ is rectifiable, giving $\area(U(t_0)\setminus f(I,t_0))=\area(U(t_0))$. Thus \eqref{eq:containments} implies
\[ \area(\limsup_i U(s_i))=\area(U(t_0))=\area(\liminf_i U(s_i)). \]
Since $\area(D)$ is finite, Lemma~\ref{lem:measures} implies $\limsup_i\area(U(s_i))\le\area(\limsup_i U(s_i))$ and $\area(\liminf_i U(s_i))\le\liminf_i\area(U(s_i))$, giving 
\[ \limsup_i\area(U(s_i))\le\area(U(t_0))\le\liminf_i\area(U(s_i)). \]
Hence $\lim_i\area(U(s_i))=\area(U(t_0))$ as required.
\end{proof}

\begin{remark}
The function $\area (U(t))$ need not be right continuous. Indeed, consider a sensor curve as shown below in Figure~\ref{fig:failed_sweep}, where $\area(U(t_0))>0$, and where there is some $\epsilon_0>0$ such that for all $0<\epsilon<\epsilon_0$, only one point on the sensor curve at time $t_0+\epsilon$ intersects $\partial D$ and $\area(U(t_0+\epsilon))=0$.

\begin{figure}[h]
	\def\svgwidth{6.5in}
	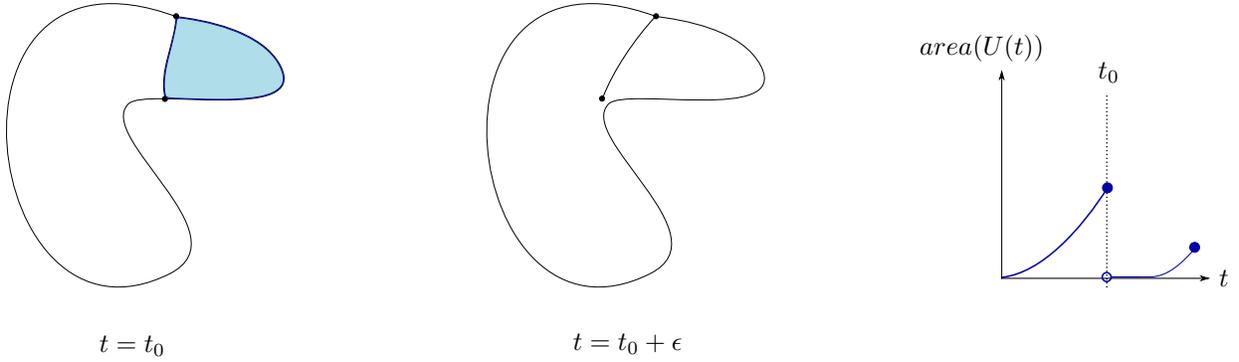
\caption{An example sensor curve where the function $\area(U(t))$ is not right continuous. The shaded region is $U(t)$ and the unshaded region is $C(t)$.}
\label{fig:failed_sweep}
\end{figure}
\end{remark}

As a consequence we obtain the following lower bound on the sweeping cost.

\begin{theorem}\label{thm:area-lower-bound}
If $D$ is a Jordan domain, then the sweeping cost $\SC(D)$ is at least as large as the length of the shortest area-bisecting curve in $D$.
\end{theorem}

\begin{proof}
Suppose that $f$ is a sweep of $D$. Note that $\area(U(0))=0$ and $\area(U(1))=\area(D)$. By Lemmas~\ref{lem:gen-intermediate-value} and \ref{lem:uppersemi-left-cont}, there exists some time $t'\in I$ with $\displaystyle \area(U(t'))=\tfrac{1}{2}\area(D)$. So $L(f(\cdot,t'))$ and hence $\SC(D)$ is at least as large as the shortest area-bisecting curve in $D$.
\end{proof}

\begin{example}\label{example:disk}
If $D = \{(x,y)\in\R^2~|~x^2+y^2\le1\}$ is the unit disk, then $\SC(D)=2$.
\end{example}

\begin{proof}
To see $\SC(D)\le2$, consider the sweep $f\colon[-1,1]\times I\to D$ defined by $f(s,t)=(2s \sqrt{t-t^2},2t-1)$ (Figure~\ref{fig:disk-sweep}) which has length $2$.

\begin{figure}[h]
	\def\svgwidth{6.5in}
	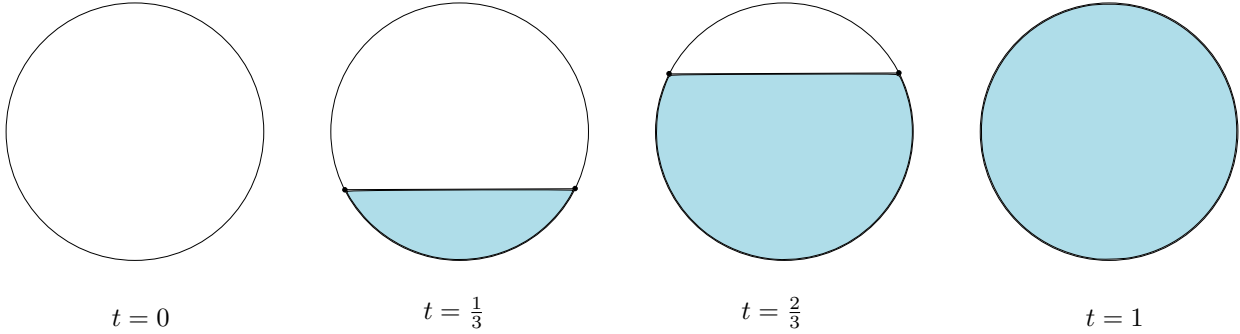
\caption{A sweep of the unit disk. The shaded region is $U(t)$ and the unshaded region is $C(t)$.}
\label{fig:disk-sweep}
\end{figure}

For the reverse direction, note that the shortest area-bisecting curve in $D$ is a diameter \cite{esposito2012longest}. Hence we apply Theorem~\ref{thm:area-lower-bound} to get $\SC(D)\ge2$.
\end{proof}

\begin{example}
Let $a,b>0$. If $D=\{(x,y)\in\R^2~|~(x/a)^2+(y/b)^2\le1\}$ is the convex hull of an ellipse, then $\SC(D)=\min\{2a,2b\}$.
\end{example}

\begin{proof}
To see $\SC(D)\le\min\{2a,2b\}$, construct a sweep much like in Example~\ref{example:disk}.

For the reverse direction, the solution to \cite[Chapter~X, Problem~33]{polya1965mathematics} states that because the ellipse has a center of symmetry, the shortest area-bisecting curve is a straight line. All area-bisecting lines pass through the center of the ellipse, and hence have length at least $\min\{2a,2b\}$.
It follows from Theorem~\ref{thm:area-lower-bound} that $\SC(D)\ge\min\{2a,2b\}$.
\end{proof}


\section{A lemma of no progress}

In Sections~\ref{sec:jordan}--\ref{sec:extremal} we will restrict attention to sensor curves $f$ with boundary points $f(0,t),f(1,t)\in\partial D$ for all $t\in I$. The motivation behind this assumption is Lemma~\ref{lem:no-progress}, which states that if $f(0,t)\notin\partial D$ or $f(1,t)\notin\partial D$, then the uncontaminated region at time $t$ is as small as possible, namely $U(t)=f(I,t)$.

The following lemma is from \cite{zoretti1905fonctions}; see also its statement in \cite[page 164]{kline1928separation}.

\begin{lemma}[Zoretti]\label{lem:zoretti}
If $K$ is a bounded maximal connected subset of a plane closed set $M$ and $\epsilon>0$, then there exists a simple closed curve $J$ enclosing $K$ such that $J\cap M=\emptyset$ and $J\subseteq B(K,\epsilon)$.
\end{lemma}

\begin{lemma}\label{lem:no-progress}
Let $f\colon I\times I\to D$ be a sensor curve. If $f(0,t)\notin\partial D$ or $f(1,t)\notin\partial D$ and $U(t)\neq D$, then $U(t)=f(I,t)$.
\end{lemma}

\begin{proof}
Without loss of generality suppose $f(0,t)\notin \partial D$. It suffices to show that $D\setminus f(I,t)$ is a single path-connected component, because then Lemma~\ref{lem:connected} and the fact that $U(t)\neq D$ will imply $C(t)=D\setminus f(I,t)$ and hence $U(t)=f(I,t)$. Let $x,x'\in D\setminus f(I,t)$; we must find a path in $D\setminus f(I,t)$ connecting $x$ and $x'$. There are two cases: when $f(1,t)\notin\partial D$, and when $f(1,t)\in\partial D$.

In the first case $f(1,t)\notin\partial D$, note that $f(I,t)$ is disjoint from $\partial D$. Hence by compactness there exists some $\epsilon>0$ such that $d(f(I,t),\partial D\cup\{x,x'\})<\epsilon$. By Lemma~\ref{lem:zoretti} (with $K=f(I, t)$ and $M=\partial D\cup\{x,x'\}$), there exists a simple closed curve $J$ in $D$ enclosing $f(I, t)$ but not enclosing $x$ or $x'$. We may therefore connect $x$ and $x'$ by a path in $D\setminus f(I,t)$ consisting of three pieces: a path in $D$ from $x$ to $J$, a path in $D$ from $x'$ to $J$, and a path in $J$ connecting these two endpoints (Figure~\ref{fig:no-progress-1}).

\begin{figure}[h]
	\def\svgwidth{6.5in}
	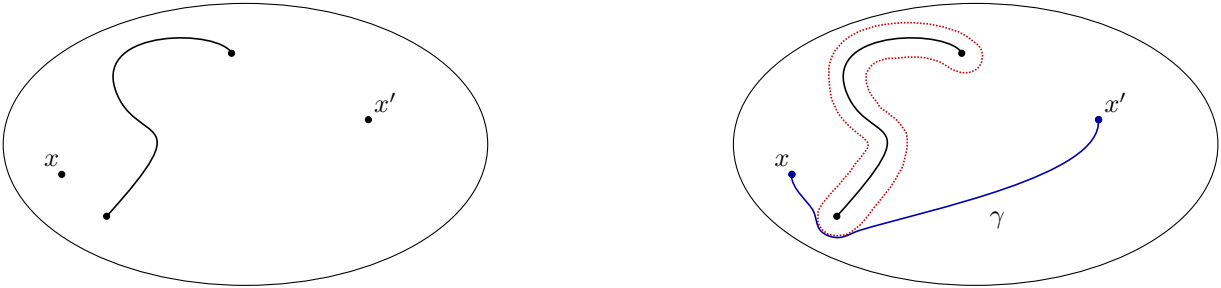
\caption{The first case in the proof of Lemma~\ref{lem:no-progress}, with $J$ drawn in red.}
\label{fig:no-progress-1}
\end{figure}

In the second case $f(1,t)\in\partial D$, pick some point $y\in D\setminus(f(I,t)\cup\{x,x'\})$. By translating $D$ in the plane we may assume that $y=\vec{0}$. Define the inversion function $i\colon\R^2\setminus\{\vec{0}\}\to\R^2\setminus\{\vec{0}\}$ by $i(r\cos\theta,r\sin\theta)=(\frac{1}{r}\cos\theta,\frac{1}{r}\sin\theta)$; note $i^2$ is the identity map. Let $\epsilon>0$ be such that
\[d(i(f(I, t)\cup \partial D),\{i(x),i(x')\})<\epsilon.\] By Lemma~\ref{lem:zoretti} (with $K=i(f(I, t)\cup \partial D)$ and $M=\{i(x),i(x')\}$), there exists a simple closed curve $J$ in $\R^2$ enclosing $i(f(I, t)\cup \partial D)$ but not enclosing $i(x)$ or $i(x')$. By the Jordan curve theorem, $i(x)$ and $i(x')$ are in the same (exterior) connected component $E$ of $\R^2\setminus J$. Since $E$ is open it is also path-connected,
and hence we can connect $i(x)$ and $i(x')$ by a path $\gamma$ in $E$. The path $i(\gamma)$ is therefore a path in $D\setminus f(I,t)$ connecting $x$ and $x'$ (Figure~\ref{fig:no-progress-2}).

\begin{figure}[h]
	\def\svgwidth{6.5in}
	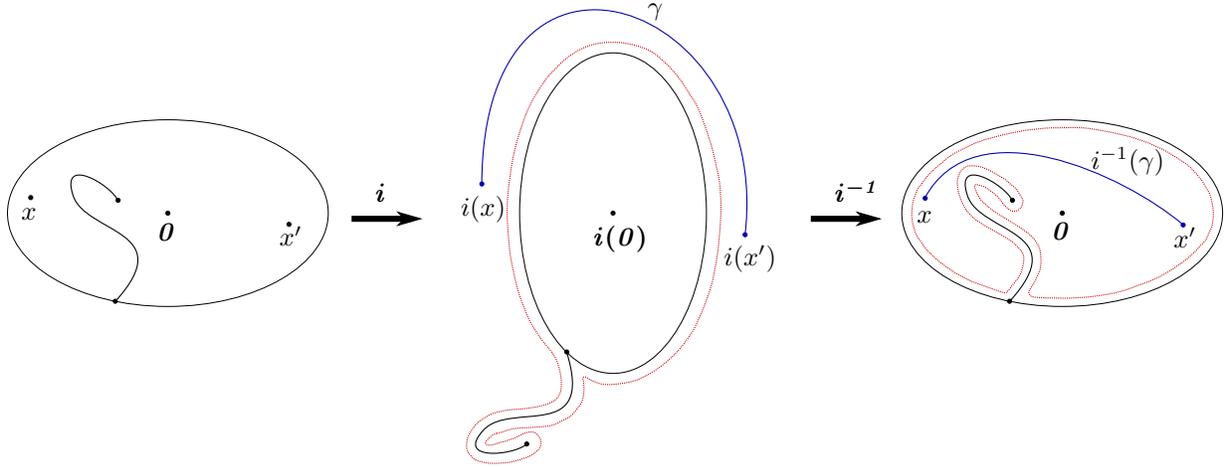
\caption{The second case in the proof of Lemma~\ref{lem:no-progress}, with $J$ and $i(J)$ drawn in red.}
\label{fig:no-progress-2}
\end{figure}
\end{proof}


\section{Sweeping cost of a Jordan domain}\label{sec:jordan}

As motivated by Lemma~\ref{lem:no-progress}, for the remainder of the paper we restrict attention to sensor curves satisfying $f(s,t)\in\partial D$ if and only if $s\in\{0,1\}$ or $t\in\{0,1\}$.

Given curves $\alpha,\beta\colon I\to\partial D$ with $\alpha(1)=\beta(0)$ (see Figure~\ref{fig:alpha_beta_winding}), we define the concatenated curve $\alpha\cdot\beta\colon I\to\partial D$ by
\[ \alpha\cdot\beta(t)=\begin{cases}
\alpha(2t) &\mbox{if }0\le t\le \frac{1}{2}\\
\beta(2t-1) &\mbox{if }\frac{1}{2}<t\le 1.
\end{cases} \]
We define the inverse curve $\beta^{-1}\colon I\to\partial D$ by $\beta^{-1}(t)=\beta(1-t)$.

\begin{figure}[h]
	\def\svgwidth{1.5in}
	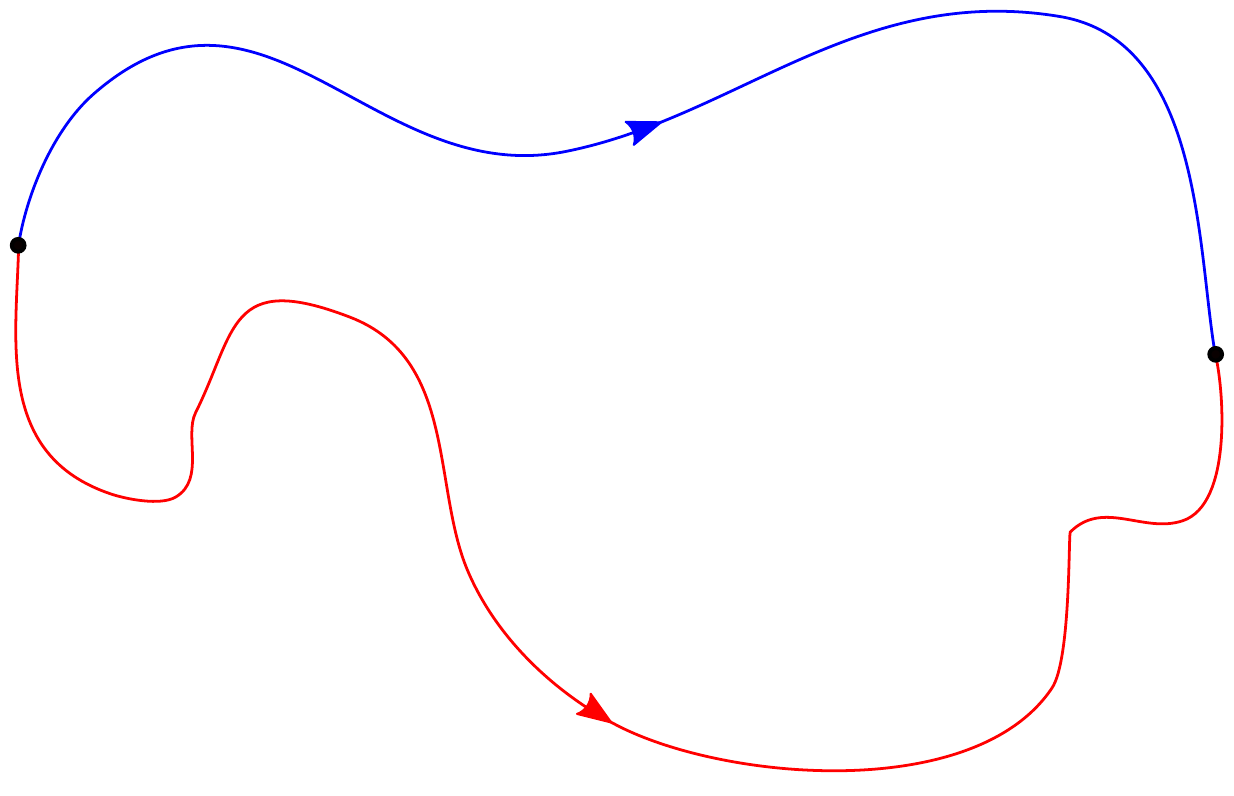
	\hspace{15mm}
	\def\svgwidth{4in}
	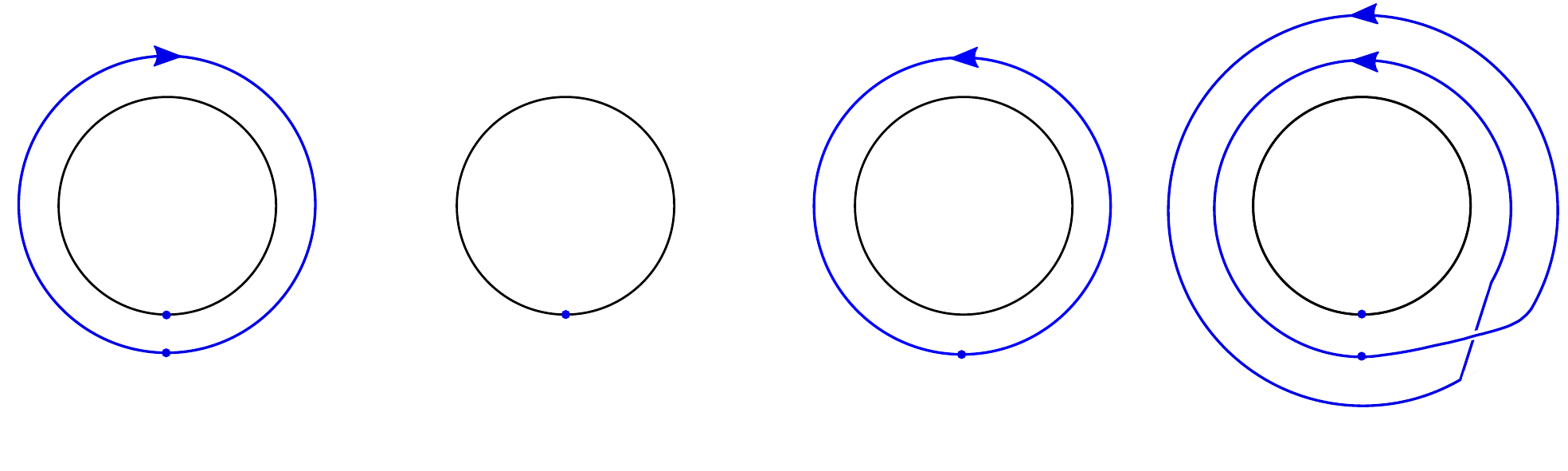
\caption{(Left) Two curves $\alpha,\beta\colon I\to\partial D$ with $\alpha(1)=\beta(0)$. (Right) Example winding numbers.}
\label{fig:alpha_beta_winding}
\end{figure}

Since $\partial D$ is homeomorphic to the circle, given a loop $\gamma\colon I\to\partial D$ (with $\gamma(0)=\gamma(1)$) we can denote the winding number of $\gamma$, i.e.\ the number of times $\gamma$ wraps around $\partial D$, by $\wn(\gamma)$. The winding number is positive (resp.\ negative) for loops that wrap around in the counterclockwise (resp.\ clockwise) direction. Note that if $\alpha$ and $\beta$ are paths in $\partial D$ with $\alpha(0)=\beta(0)$ and $\alpha(1)=\beta(1)$, then $\alpha\cdot\beta^{-1}$ is a loop.

Our main result is an analytic formula for the sweeping cost of a Jordan domain.

\begin{theorem}\label{thm:frechet} 
The sweeping cost of a Jordan domain $D$ is
\begin{equation}\label{eq:frechet}
\SC(D)=\inf\{d_L(\alpha,\beta)~|~\alpha,\beta\colon I\to\partial D,\ \alpha(0)=\beta(0),\ \alpha(1)=\beta(1),\ \wn(\alpha\cdot\beta^{-1})\neq0\}.
\end{equation}
\end{theorem}

Equation~\eqref{eq:frechet} is closely related to the \emph{geodesic width} between two polylines \cite{efrat2002new}, and also the \emph{isotopic Fr\'{e}chet distance} between two curves \cite{chambers2011isotopic}. Indeed, note that the right hand side of \eqref{eq:frechet} is unchanged if we replace $d_L(\alpha,\beta)$ with the weak geodesic Fr\'{e}chet distance
between $\alpha$ and $\beta$ (see Definition~\ref{def:geodesic-frechet}).

\begin{remark}\label{rem:frechet-unchanged}
The value of the right hand side of \eqref{eq:frechet} is unchanged if we replace $\wn(\alpha\cdot\beta^{-1})\neq0$ with $\wn(\alpha\cdot\beta^{-1})\in\{-1,1\}$.
\end{remark}

\begin{proof}[Proof of Remark~\ref{rem:frechet-unchanged}]
Let $\alpha,\beta\colon I\to\partial D$ with $\alpha(0)=\beta(0)$ and $\alpha(1)=\beta(1)$, and suppose $|\wn(\alpha\cdot\beta^{-1})|\ge2$.
Hence there exists some $0<t<1$ such that $\alpha(t)=\beta(t)$ and $\wn(\alpha|_{[0,t]}\cdot\beta|_{[0,t]}^{-1})\in\{-1,1\}$.
The claim follows since
\[ d_L(\alpha|_{[0,t]},\beta|_{[0,t]})\le d_L(\alpha,\beta). \]
\end{proof}

The following lemma will be used to prove the $\le$ direction in \eqref{eq:frechet}.

\begin{lemma}\label{lem:frechet-le}
Let $D$ be a Jordan domain. Suppose $\alpha,\beta\colon I\to\partial D$ with $\alpha(0)=\beta(0)$, $\alpha(1)=\beta(1)$, and $\wn(\alpha\cdot\beta^{-1})\neq 0$. If $f \colon I \times I \to D$ is any sensor curve with $f(0,t)=\alpha(t)$ and $f(1,t)=\beta(t)$, then $f$ is a sweep of $D$.
\end{lemma}

\begin{proof}
Let $\D=\{(x,y)\in\R^2~|~x^2+y^2\le1\}$ be the unit disk. We first prove this claim in the case when $D=\D$.

By Lemma~\ref{lem:disk-curves}, there exists a point $p\in\R^2\setminus\D$ and two continuous families of curves $g_\alpha,g_\beta\colon I\times I\to\R^2$ such that
\begin{itemize}
\item $g_\alpha(0,t)=p=g_\beta(0,t)$,
\item $g_\alpha(1,t)=\alpha(t)$,
\item $g_\beta(1,t)=\beta(t)$, and 
\item $g_\alpha(s,t),g_\beta(s,t)\notin\D$ for $s<1$.
\end{itemize}
Let $S^1$ be the circle of unit circumference, i.e.\ $[0,1]$ with endpoints $0$ and $1$ identified. Define a continuous map $g\colon S^1\times I\to D$ via
\[g(s,t)=\begin{cases}
g_\alpha(3st,t)&\mbox{if }0\le s<\frac{1}{3}\\
f(3s-1,t)&\mbox{if }\frac{1}{3}\le s<\frac{2}{3}\\
g_\beta(3t(1-s),t)&\mbox{if }\frac{2}{3}\le s\le 1.\\
\end{cases}\]
Note $g(\cdot,t)$ is indeed a (possibly non-simple) map from the circle since $g_\alpha(0,t)=p=g_\beta(0,t)$ for all $t$. Define a continuous signed distance $d^\pm \colon \R^2 \times I \to \R$ by 
\begin{equation*}
d^\pm(x,t)=\begin{cases}
d(x,g(S^1,t)) &\mbox{if }x\in g(S^1,t) \mbox{ or }\wn(g(\cdot,t),x)=0 \\
-d(x,g(S^1,t)) &\mbox{if }x\notin g(S^1,t) \mbox{ and }\wn(g(\cdot,t),x)\neq0.
\end{cases}
\end{equation*}
Here $\wn(g(\cdot,t),x)$ denotes (for $x\notin g(S^1,t)$) the winding number of the map $g(\cdot,t)\colon S^1\to\R^2\setminus\{x\}\simeq S^1$. Note that $\wn(g(\cdot,t),x)$ is constant on each connected component of $\R^2\setminus g(S^1,t)$, and that $d^\pm$ is continuous.

Given any intruder path $\gamma\colon I\to\D$, the continuous function $d^\pm(\gamma(t),t) \colon I \to \R$ satisfies $d^\pm(\gamma(0),0) \ge 0$ (since $f(I,0)$ is a single point in $\partial D$) and $d^\pm(\gamma(1),1) \le 0$ (since $f(I,1)$ is a single point in $\partial D$ and $\wn(\alpha\cdot\beta^{-1})\neq 0$). By the intermediate value theorem there exists some $t' \in I$ with $d^\pm(\gamma(t'),t')=0$, and hence $\gamma(t')\in g(S^1,t')\cap\D=f(I,t')$. So $\gamma$ is not an evasion path, and $f$ is a sweep of $\D$.

We now handle the case when $D$ is an arbitrary Jordan domain. By definition there exists a homeomorphism $h\colon\D\to D$. Note that $h^{-1}\alpha,h^{-1}\beta\colon I\to\partial \D$ with $h^{-1}\alpha(0)=h^{-1}\beta(0)$, $h^{-1}\alpha(1)=h^{-1}\beta(1)$, and $\wn(h^{-1}\alpha\cdot h^{-1}\beta^{-1})\neq 0$. Since $h^{-1}f \colon I \times I \to \D$ is a sensor curve, it follows from our proof in the case of the disk that $h^{-1}f$ is a sweep of $\D$. Hence $f$ is a sweep of $D$ by Lemma~\ref{lem:homeo}.
\end{proof}

\begin{proof}[Proof of Theorem~\ref{thm:frechet}]
Let
\[ c = \inf\{d_L(\alpha,\beta)~|~\alpha,\beta\colon I\to\partial D,\ \alpha(0)=\beta(0),\ \alpha(1)=\beta(1),\ \wn(\alpha\cdot\beta^{-1})\neq0\}. \]
We first prove the $\le$ direction of \eqref{eq:frechet}. Let $\epsilon>0$ be arbitrary. By the definition of infimum there exist curves $\alpha,\beta\colon I\to\partial D$ with $\alpha(0)=\beta(0)$, $\alpha(1)=\beta(1)$, $\wn(\alpha\cdot\beta^{-1})\neq0$, and $d_L(\alpha,\beta)\le c+\epsilon$. Define $f\colon I\times I\to D$ by letting $f(\cdot, t)\colon I\to D$ be the unique constant-speed geodesic in $D$ between $f(0,t)=\alpha(t)$ and $f(1,t)=\beta(t)$, which exists by \cite{bourgin1989shortest,bishop2005intrinsic}. 
Lemma~\ref{lem:frechet-le} implies that $f$ is a sweep, and hence we have
\[ \SC(D)\le L(f)=d_L(\alpha,\beta)\le c+\epsilon. \]
Since this is true for all $\epsilon>0$, we have $\SC(D)\le c$.

For the $\ge$ direction of \eqref{eq:frechet}, suppose that $f$ is a sensor curve with $L(f)<c$. For notational convenience, define $\alpha,\beta\colon I\to\partial D$ by $\alpha(t)=f(0,t)$ and $\beta(t)=f(1,t)$. Then necessarily $\wn(\alpha\cdot\beta^{-1})=0$, and furthermore
\begin{equation}\label{eq:partial-wn}
\alpha(t)=\beta(t)\quad\mbox{implies}\quad\wn(\alpha|_{[0,t]}\cdot\beta|_{[0,t]}^{-1})=0,
\end{equation}
since otherwise we'd have
\[ L(f)\ge d_L(\alpha,\beta) \ge d_L(\alpha|_{[0,t]},\beta|_{[0,t]})\ge c, \]
a contradiction. We will show that $f$ is not a sweep of $D$ by showing the existence of an evasion path $\gamma\colon I\to D$ whose image furthermore lives in $\partial D$.

Indeed, consider the 1-dimensional evasion problem in $\partial D$ where the region covered by the sensors at time $t$ is $\{\alpha(t),\beta(t)\}$. In this 1-dimensional problem, it is clear that the uncontaminated region in $\partial D$ is either (i) a single point $\alpha(t)=\beta(t)$, (ii) a closed interval in $\partial D$ with endpoints $\alpha(t)$ and $\beta(t)$, or (iii) all of $\partial D$. Equation~\eqref{eq:partial-wn}, however, rules out the possibility of (iii). It follows that the contaminated region in $\partial D$ is always a nonempty open interval in $\partial D$ with continuously varying endpoints $\alpha(t)$ and $\beta(t)$. Therefore we can define an evasion path $\gamma\colon I\to \partial D$, for example by letting $\gamma(t)$ be the midpoint of the open interval of the uncontaminated region in $\partial D$. This evasion path $\gamma$ is also an evasion path for our original 2-dimensional problem in $D$, as $\gamma\colon I\to \partial D\subseteq D$ satisfies $\gamma(t)\notin f(I,t)$ for all $t$. This gives the $\ge$ direction of \eqref{eq:frechet}.
\end{proof}

\begin{question}
Does Theorem~\ref{thm:frechet} hold even if assumption (iii) in Definition~\ref{def:sensor-curve} is removed, i.e.\ if the interior of a sensor curve is also allowed to touch $\partial D$?
\end{question}

\begin{question}\label{ques:weak-strong}
For any Jordan domain $D$ in the plane and injective curves $\alpha,\beta\colon I\to\partial D$, 
we conjecture that the weak geodesic Fr\'{e}chet distance between $\alpha$ and $\beta$ is equal to their strong geodesic Fr\'{e}chet geodesic distance.
\end{question}

There are simple counterexamples to Question~\ref{ques:weak-strong} when $\alpha$ and $\beta$ are not injective, or when they do not map to $\partial D$. Closely related is following question: is the value of \eqref{eq:frechet} unchanged if we require $\alpha$ and $\beta$ to be injective?

\section{Sweeping cost of a convex domain}\label{sec:convex}

Given a convex Jordan domain $D\subseteq \R^2$, its \emph{width} $w(D)$ is defined as
\[ w(D) = \min_{\|v\|=1}\ \max_{x\in\R^2}\ L(D\cap\{x+tv~|~t\in\R\}), \]
where $v$ is a unit direction vector in $\R^2$. Alternatively, the width $w(D)$ is the smallest distance between two parallel supporting lines on opposite sides of $D$ (Figure~\ref{fig:width}).

\begin{figure}[h]
	\def\svgwidth{2in}
	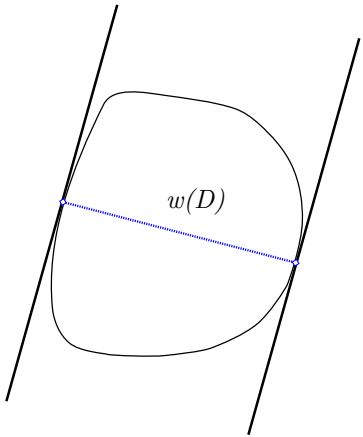
\caption{The width $w(D)$ of a domain $D$.}
\label{fig:width}
\end{figure}

\begin{theorem}\label{thm:sc=width}
If $D$ is a convex Jordan domain, then $\SC(D)=w(D)$.
\end{theorem}

\begin{proof}
We first show $\SC(D)\le w(D)$. By Theorem~\ref{thm:frechet}, it suffices to show
\[ \inf\{d_L(\alpha,\beta)~|~\alpha,\beta\colon I\to\partial D,\ \alpha(0)=\beta(0),\ \alpha(1)=\beta(1),\ \wn(\alpha\cdot\beta^{-1})\neq0\} \le w(D). \]
Let $v$ some direction vector realizing the width, i.e.\ $w(D)=\max_{x\in\R^2}\ L(D\cap\{x+tv~|~t\in\R\})$. Consider sweeping through all lines in $\R^2$ parallel to $v$; the intersection of these lines with $\partial D$ traces out two continuous curves $\alpha,\beta\colon I\to\partial D$ with $\alpha(0)=\beta(0)$, $\alpha(1)=\beta(1)$, and $\wn(\alpha\cdot\beta^{-1})=\pm 1$.\footnote{This is not quite precise if $\partial D$ contains a straight line segment of non-zero length parallel to $v$ (there are at most two such segments). In this case, pick an arbitrary point on each such line segment; each such point will be either the starting point or the ending point for both $\alpha$ and $\beta$.} We have $d_L(\alpha,\beta)\le w(D)$, giving $\SC(D)\le w(D)$.

To finish the proof, we need some background on planar convex domains. A point $x\in\partial D$ is \emph{smooth} if it has a unique supporting hyperplane, and otherwise $x$ is a \emph{vertex} containing a range of angles $[\theta_1,\theta_2]\subseteq S^1$ (with $\theta_1\neq\theta_2$) which are the outward normal directions of supporting hyperplanes of $D$ at $x$. Away from the vertices, the unique supporting hyperplane of $x\in\partial D$ varies continuously with $x$. By \cite[Proposition~11.6.2]{berger2009geometry}, 
the set of vertices of the closed convex domain $D$ is countable.

We now show $\SC(D)\ge w(D)$. Given $\epsilon>0$, let $\alpha$ and $\beta$ be $\epsilon$-close to realizing the infimum in \eqref{eq:frechet}, meaning $\SC(D)+\epsilon\ge d_L(\alpha,\beta)$. For notational convenience we assume that $\alpha(0)=\beta(0)$ and $\alpha(1)=\beta(1)$ are not vertices of $\partial D$ (our same proof technique works regardless). Let $T=\{t_1,t_2,t_3,\ldots\}\subseteq I$ be a countable subset such that $t\in T$ if either $\alpha(t)$ or $\beta(t)$ is a vertex of $\partial D$. Let $t_0=0$, and if $|T|$ is finite, then let $t_{|T|+1}=1$. For $i=1,2,\ldots,|T|$, choose weights $w_i>0$ such that $\sum_i w_i=w<\infty$; this is possible since $T$ is countable. Let $p_1\colon \partial D \times S^1\to \partial D$ and $p_2\colon \partial D\times S^1\to S^1$ be the projection maps. It is possible to define a \emph{continuous} map $g_\alpha\colon [0,1+w]\to \partial D\times S^1$ satisfying the following properties.
\begin{itemize}
\item Each $g_\alpha(s)$ is equal to a point $(\alpha(t),v)\in\partial D\times S^1$ with $t\in I$ such that $v$ is the outward normal vector to a supporting hyperplane of $D$ at $\alpha(t)$.
\item If $t_i\le t\le t_{i+1}$, then $\alpha(t)=p_1g_\alpha(t+\sum_{j=1}^iw_j)$.
\item For all $0\le s\le w_i$, we have $p_1g_\alpha(s+t_i+\sum_{j=1}^{i-1}w_j)=\alpha(t_i)$.
\item As $s$ varies from $0$ to $w_i$, angle $p_2g_\alpha(s+t_i+\sum_{j=1}^{i-1}w_j)$ varies over the range of supporting hyperplanes of $D$ at $\alpha(t_i)$ (which may be a single angle if $\alpha(t_i)$ is not a vertex of $\partial D$).
\end{itemize}
Define $g_\beta\colon [0,1+w]\to \partial D\times S^1$ similarly (with $\alpha$ replaced everywhere by $\beta$). Note that $g_\alpha(0)=g_\beta(0)$, that $g_\alpha(1+w)=g_\beta(1+w)$, and that $p_2g_\alpha$ and $p_2g_\beta$ wrap in opposite directions around $S^1$. Hence for some $s\in[0,1+w]$ the supporting hyperplanes corresponding to $g_\alpha(s)$ and $g_\beta(s)$ will be parallel and on opposite sides of $D$. It follows that
\[ \SC(D)+\epsilon\ge d_L(\alpha,\beta)\ge d(p_1g_\alpha(s),p_1g_\beta(s))\ge w(D). \]
Since this is true for all $\epsilon>0$, we have $\SC(D)\ge w(D)$.
\end{proof}

The paper \cite{houle1988computing} shows that for $D\subseteq \R^2$ a convex polygonal domain with $n$ vertices, the width and hence the sweeping cost of $D$ can be computed in time $O(n)$ and space $O(n)$ using the \emph{rotating calipers} technique.

\section{Extremal shapes}\label{sec:extremal}

Which convex shape of unit area has the largest sweeping cost? The papers \cite[Theorem~4.3]{cerdan2006comparing} and \cite{pal1921minimumproblem} state that if $D$ is a bounded planar convex domain, then $\area(D)\ge w(D)^2/\sqrt{3}$,
where equality is achieved if $D$ is an equilateral triangle. The next corollary follows immediately from Theorem~\ref{thm:sc=width}.

\begin{corollary}
Let $D$ be a convex Jordan domain. Then
\[ \area(D)\ge\frac{\SC(D)^2}{\sqrt{3}}, \]
where equality is achieved if $D$ is an equilateral triangle. Hence the equilateral triangle has the maximal sweeping cost over all planar convex domains of the same area.
\end{corollary}

The next example shows that there is no extremal shape for non-convex Jordan domains.

\begin{example}
A (non-convex) domain $D$ of unit area may have arbitrarily large sweeping cost.
\end{example}

\begin{proof}
Consider a deformation of an equilateral triangle with unit side lengths where we deform each edge towards the center of the triangle (Figure~\ref{fig:triangle-deform}). Note that as the sweeping cost converges to $\tfrac{1}{\sqrt{3}}$ (the distance from the center to a vertex) from above, the area of the shape tends zero. Rescaling each shape in this deformation to have area one shows that a non-convex domain of unit area may have arbitrarily large sweeping cost.

\begin{figure}[h]
	\def\svgwidth{6.5in}
	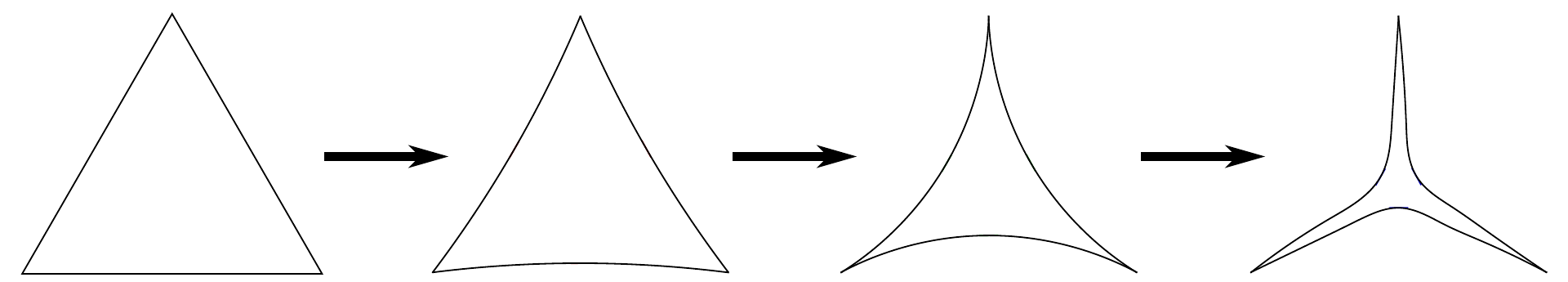
\caption{We deform each edge of the triangle towards the center of the triangle, producing a three-pronged shape. As the sweeping costs of the shapes converge to a fixed constant, the areas converge to zero.}
\label{fig:triangle-deform}
\end{figure}
\end{proof}

\section{Conclusion}\label{sec:conclusion}

Given a Jordan domain $D$ in the plane, we show that the sweeping cost of $D$ is at least as large as the shortest area-bisecting curve in $D$, and we give a formula for the sweeping cost in terms of the geodesic Fr\'{e}chet distance between two curves on the boundary of $D$ with non-equal winding numbers. We show that the sweeping cost of any convex domain is equal to its width. Therefore, the sweeping cost of a polygonal convex domain with $n$ vertices can be computed in time and space $O(n)$, and a convex domain of unit area with maximal sweeping cost is the equilateral triangle.

We end by mentioning two related settings of interest. First, let $D$ be a compact region in the plane, perhaps not simply-connected. Suppose the pursuer is now a union of curves. What can one say about the sweeping cost of $D$, measured as the sum of the curve lengths? Second, let $D\subset \R^n$ be the homemorphic image of the closed $n$-dimensional ball. What are the properties of the sweeping cost of $D$, when swept by an $(n-1)$-dimensional ``sensor surface"? For example, what is a 3-dimensional convex body of unit volume which maximizes this higher-dimensional sweeping cost?

\section{Acknowledgements}

We would like to thank Clayton Shonkwiler for pointing us to the references \cite{cerdan2006comparing,pal1921minimumproblem}.

\bibliographystyle{plain}
\bibliography{SweepingCostsOfPlanarDomains}

\begin{thebibliography}{10}

\bibitem{adams2015evasion}
Henry Adams and Gunnar Carlsson.
\newblock Evasion paths in mobile sensor networks.
\newblock {\em The International Journal of Robotics Research}, 34(1):90--104,
  2015.

\bibitem{alonso1992lion}
Laurent Alonso, Arthur~S Goldstein, and Edward~M Reingold.
\newblock ``{L}ion and man": {U}pper and lower bounds.
\newblock {\em ORSA Journal on Computing}, 4(4):447--452, 1992.

\bibitem{alspach2006searching}
Brian Alspach.
\newblock Searching and sweeping graphs: {A} brief survey.
\newblock {\em Le matematiche}, 59:5--37, 2006.

\bibitem{berger2009many}
Florian Berger, Alexander Gilbers, Ansgar Gr{\"u}ne, and Rolf Klein.
\newblock How many lions are needed to clear a grid?
\newblock {\em Algorithms}, 2(3):1069--1086, 2009.

\bibitem{berger2009geometry}
Marcel Berger.
\newblock {\em Geometry {I}}.
\newblock Springer Science \& Business Media, 2009.

\bibitem{beveridge2015two}
Andrew Beveridge and Yiqing Cai.
\newblock Two-dimensional pursuit-evasion in a compact domain with piecewise
  analytic boundary.
\newblock {\em arXiv preprint 1505.00297}, 2015.

\bibitem{bienstock1991graph}
Daniel Bienstock.
\newblock Graph searching, path-width, tree-width and related problems (a
  survey).
\newblock {\em DIMACS Series in Discrete Mathematics and Theoretical Computer
  Science}, 5:33--49, 1991.

\bibitem{bishop2005intrinsic}
Richard~L Bishop.
\newblock The intrinsic geometry of a {J}ordan domain.
\newblock {\em arXiv preprint math/0512622}, 2005.

\bibitem{bourgin1989shortest}
Richard~D Bourgin and Peter~L Renz.
\newblock Shortest paths in simply connected regions in $\mathbb{R}^2$.
\newblock {\em Advances in Mathematics}, 76(2):260--295, 1989.

\bibitem{burago2001course}
Dmitri Burago, Yuri Burago, and Sergei Ivanov.
\newblock {\em A course in metric geometry}, volume~33.
\newblock American Mathematical Society, Providence, 2001.

\bibitem{cerdan2006comparing}
A~Cerd{\'a}n.
\newblock Comparing the relative volume with the relative inradius and the
  relative width.
\newblock {\em Journal of Inequalities and Applications}, 2006(1):1--8, 2006.

\bibitem{chambers2011isotopic}
Erin~W Chambers, David Letscher, Tao Ju, and Lu~Liu.
\newblock Isotopic fr{\'e}chet distance.
\newblock In {\em CCCG}, 2011.

\bibitem{chin2010detection}
Jren-Chit Chin, Yu~Dong, Wing-Kai Hon, Chris Yu-Tak Ma, and David~KY Yau.
\newblock Detection of intelligent mobile target in a mobile sensor network.
\newblock {\em IEEE/ACM Transactions on Networking (TON)}, 18(1):41--52, 2010.

\bibitem{chung2011search}
Timothy~H Chung, Geoffrey~A Hollinger, and Volkan Isler.
\newblock Search and pursuit-evasion in mobile robotics.
\newblock {\em Autonomous robots}, 31(4):299--316, 2011.

\bibitem{cohn1980measure}
Donald~L Cohn.
\newblock {\em Measure Theory}, volume 165.
\newblock Springer, 1980.

\bibitem{cortes2002coverage}
Jorge Cortes, Sonia Martinez, Timur Karatas, and Francesco Bullo.
\newblock Coverage control for mobile sensing networks.
\newblock In {\em Proceedings of the IEEE International Conference on Robotics
  and Automation}, volume~2, pages 1327--1332, 2002.

\bibitem{de2006coordinate}
Vin de~Silva and Robert Ghrist.
\newblock Coordinate-free coverage in sensor networks with controlled
  boundaries via homology.
\newblock {\em The International Journal of Robotics Research},
  25(12):1205--1222, 2006.

\bibitem{efrat2002new}
Alon Efrat, Leonidas~J Guibas, Sariel Har-Peled, Joseph~SB Mitchell, and
  TM~Murali.
\newblock New similarity measures between polylines with applications to
  morphing and polygon sweeping.
\newblock {\em Discrete \& Computational Geometry}, 28(4):535--569, 2002.

\bibitem{esposito2012longest}
Luca Esposito, Vincenzo Ferone, Bernd Kawohl, Carlo Nitsch, and Cristina
  Trombetti.
\newblock The longest shortest fence and sharp {P}oincar{\'e}--{S}obolev
  inequalities.
\newblock {\em Archive for Rational Mechanics and Analysis}, 206(3):821--851,
  2012.

\bibitem{houle1988computing}
Michael~E Houle and Godfried~T Toussaint.
\newblock Computing the width of a set.
\newblock {\em IEEE Transactions on Pattern Analysis and Machine Intelligence},
  10(5):761--765, 1988.

\bibitem{johnstone1989hotelling}
Iain Johnstone and David Siegmund.
\newblock On {H}otelling's formula for the volume of tubes and {N}aiman's
  inequality.
\newblock {\em The Annals of Statistics}, pages 184--194, 1989.

\bibitem{kline1928separation}
JR~Kline.
\newblock Separation theorems and their relation to recent developments in
  analysis situs.
\newblock {\em Bulletin of the American Mathematical Society}, 34(2):155--192,
  1928.

\bibitem{liu2005mobility}
Benyuan Liu, Peter Brass, Olivier Dousse, Philippe Nain, and Don Towsley.
\newblock Mobility improves coverage of sensor networks.
\newblock In {\em Proceedings of the 6th ACM international symposium on Mobile
  ad hoc networking and computing}, pages 300--308. ACM, 2005.

\bibitem{pal1921minimumproblem}
Julius P{\'a}l.
\newblock Ein minimumproblem f{\"u}r ovale.
\newblock {\em Mathematische Annalen}, 83(3):311--319, 1921.

\bibitem{parsons1978pursuit}
Torrence~D Parsons.
\newblock Pursuit-evasion in a graph.
\newblock In {\em Theory and applications of graphs}, pages 426--441. Springer,
  1978.

\bibitem{polya1965mathematics}
George P{\'o}lya.
\newblock {\em Mathematics and plausible reasoning: vol 1: Induction and
  analogy in mathematics.}
\newblock Oxford University Press, 1965.

\bibitem{sriraghavendra2007frechet}
E~Sriraghavendra, K~Karthik, and Chiranjib Bhattacharyya.
\newblock Fr{\'e}chet distance based approach for searching online handwritten
  documents.
\newblock In {\em Ninth International Conference on Document Analysis and
  Recognition}, volume~1, pages 461--465, 2007.

\bibitem{toussaint1983solving}
Godfried~T Toussaint.
\newblock Solving geometric problems with the rotating calipers.
\newblock In {\em Proc.\ IEEE Melecon}, volume~83, page A10, 1983.

\bibitem{wangmeasuring}
Yusu Wang.
\newblock Measuring similarity between curves on 2-manifolds via minimum
  deformation area.
\newblock Technical report, Technical report, Dept.\ of Comp. Sc. \& Eng., Ohio
  State Univ., 2008.

\bibitem{zoretti1905fonctions}
Ludovic Zoretti.
\newblock Sur les fonctions analytiques uniformes.
\newblock {\em J. Math. pures appl}, 1:9--11, 1905.

\end{thebibliography}

\appendix

\section{Additional lemmas and proofs}

\begin{lemma}\label{lem:measures}
If $(X,\mu)$ is a measure space and $U_i$ is a sequence of measurable sets in $X$, then
\begin{enumerate}
\item $\mu(\liminf_i U_i)\le\liminf_i\mu(U_i)$, and
\item $\mu(\limsup_i U_i)\ge\limsup_i\mu(U_i)$ if $\mu(X)<\infty$.
\end{enumerate}
\end{lemma}

\begin{proof}
Recall $\limsup_i U_i=\cap_{i=1}^\infty(\cup_{j=i}^\infty U_i)$. Since the $\cup_{j=i}^\infty U_i$ are a decreasing sequence of sets, and since $\mu(X)<\infty$, \cite[Proposition~1.2.3]{cohn1980measure} implies $\mu(\limsup_i U_i)=\lim_i\mu(\cup_{j=i}^\infty U_i)$. Since $U_i\subseteq \cup_{j=i}^\infty U_j$, we have $\mu(U_i)\leq \mu(\cup_{j=i}^\infty U_j)$, and hence 
\[ \mu(\limsup_i U_i)=\lim_i\mu(\cup_{j=i}^\infty U_i)\ge\limsup_i\mu(U_i), \]
giving (2). The proof of (1) is similar except that the finiteness assumption is unnecessary.
\end{proof}

\begin{lemma}\label{lem:disk-curves}
Let $\D=\{(x,y)\in\R^2~|~x^2+y^2\le1\}$ be the unit disk. Given any point $p\in\R^2\setminus\D$ and any curve $\alpha\colon I\to\partial\D$, there exists a continuous function $g\colon I\times I\to\R^2$ such that
\begin{itemize}
\item $g(0,t)=p$ for all $t\in I$,
\item $g(1,t)=\alpha(t)$ for all $t\in I$, and
\item $g(s,t)\notin\D$ for $s<1$.
\end{itemize}
\end{lemma}

\begin{proof}
Given $r\ge0$, let $\alpha_r(t)\colon I\to\R^2$ be defined by $\alpha_r(t)=(1+r)\alpha(t)$. Fix some $\epsilon>0$. Pick a single curve $\gamma\colon I\to\R^2\setminus\D$ with $\gamma(0)=p$ and $\gamma(1)=\alpha_\epsilon(0)$; this is possible since $\R^2\setminus\D$ is connected by the Jordan curve theorem. We define the function $g\colon I\times I\to\R^2$ as follows:
\[g(s,t) = \begin{cases}
\gamma(3s) &\mbox{if }0 \leq s < \frac{1}{3}\\
\alpha_\epsilon((3s-1)t) &\mbox{if }\frac{1}{3} \leq s < \frac{2}{3}\\
\alpha_{(3-3s)\epsilon}(t) &\mbox{if }\frac{2}{3} \leq s \leq 1.
\end{cases} \]
Note that $g$ is continuous and satisfies all of the required conditions.\end{proof}

\end{document}